\documentclass[12pt]{amsart} 
\usepackage[style=authoryear, backend=biber, style=numeric]{biblatex}
\bibliography{refs}

\usepackage{amsmath}
\usepackage{amssymb}
\usepackage{mathtools}
\usepackage[utf8]{inputenc}
\usepackage{tikz-cd}
\newtheorem{theorem}{Theorem}[section]
\newtheorem{corollary}[theorem]{Corollary}
\newtheorem{proposition}[theorem]{Proposition}
\newtheorem{lemma}[theorem]{Lemma}

\theoremstyle{definition}
\newtheorem{definition}[theorem]{Definition}
\newtheorem{example}[theorem]{Example}

\theoremstyle{remark}
\newtheorem{remark}[theorem]{Remark}
\DeclareMathOperator{\Pow}{\mathcal P}

\DeclareMathOperator{\dis}{\mathrm {dis}}
\DeclareMathOperator{\id}{\mathrm {id}}
\DeclareMathOperator{\ob}{\mathrm {ob}}

\DeclareMathOperator{\diam}{\mathrm {diam}}
\DeclareMathOperator{\intdist}{d_{\mathrm {int}}}
\DeclareMathOperator{\netwdist}{d_{\mathcal N}}

\newcommand{\xto}{\xrightarrow}
\newcommand{\setrel}{\mathcal S}
\newcommand{\relcat}{\mathcal R}

\newcommand{\op}{\mathrm{op}}
\newcommand{\RR}{\mathbb R}
\newcommand{\cech}{\check{\mathcal C}}
\newcommand{\rstar}{[0,\infty]}
\newcommand{\filrel}{[\rstar, \relcat]}

\newcommand{\CC}{\mathcal C}
\newcommand{\DD}{\mathcal D}

\newcommand{\cx}{\mathrm {Cx}}

\newcommand{\hcx}{\mathrm {hCx}}
\newcommand{\topsp}{\mathrm {Top}}
\newcommand{\htop}{\mathrm {hTop}}
\newcommand{\dow}{\mathrm {Dow}}
\newcommand{\filhtop}{[\rstar,\mathrm {hTop}]}

\begin{document}

\title{Sparse Dowker Nerves}
\author{Nello Blaser}
\author{Morten Brun}

\begin{abstract}
  We propose sparse versions of filtered simplicial
  complexes used to compte persistent homology of point clouds and of
  networks. In particular we extend a slight variation of the Sparse
  \v Cech Complex of  Cavanna, Jahanseir and Sheehy \cite{SRGeom} from
  point clouds in Cartesian space to point clouds in arbitrary metric
  spaces. Along the way we formulate interleaving in terms of strict
  \(2\)-categories, and we introduce the concept of Dowker
  dissimilarities that can be considered as a common generalization of
  metric spaces and networks.
\end{abstract}
\maketitle
\section{Introduction}
\label{sec:intro}

This paper is the result of an attempt to obtain the interleaving
guarantee for the sparse \v Cech
complex of Cavanna, Jahanseir and Sheehy \cite{SRGeom} without using
the Nerve Theorem. The 
rationale for this was to generalize the result to arbitrary
metric spaces. We have not been able to show that the constructions of 
\cite{Sheehy2013} or \cite{SRGeom} are interleaved with the \v Cech
complex in arbitrary 
metric spaces. However, changing the construction slightly, we obtain a
sub-complex of the \v Cech complex that is interleaved
in a similar way. When applied to point clouds in \(\RR^d\) with a convex
metric this sub-complex is homotopic to the construction of
\cite{SRGeom} 

The search for a more general version of the sparse \v Cech complex
led us to study both different versions of filtered covers and
extended metrics. We discovered that these concepts are instances of
filtered relations given by functions of the form 
\begin{displaymath}
  \Lambda \colon L \times W \to \rstar
\end{displaymath}
from the product of two sets \(L\) and \(W\) to the interval
\(\rstar\). Given \(t \in \rstar\), the relation \(\Lambda_t\) at
filtration level \(t\) is
\begin{displaymath}
  \Lambda_{t} = \{(l,w) \in L \times W \, \mid \, \Lambda(l,w) < t\}.
\end{displaymath}
\cite{MR0048030} observed that a relation \(R \subseteq L \times
W\) gives a cover \((R(l))_{l \in L}\) of the set
\begin{displaymath}
  R_W = \{w \in W \, \mid \, 
  \text{ there exists \(l\in L\) with \((l,w) \in R\)} \}
\end{displaymath}
with
\begin{displaymath}
  R(l) = \{w \in W \, \mid \, 
  (l,w) \in R \}.
\end{displaymath}
The {\em Dowker complex} of the relation \(R\) is the Borsuk Nerve of this
cover. The {\em Dowker Homology Duality Theorem} \cite[Theorem 1]{MR0048030}
states that the Dowker complexes of \(R\) and the transposed relation
\begin{displaymath}
  R^t = \{(w,l) \, \mid \, (l, w)\in R\} \subseteq W \times L
\end{displaymath}
have isomorphic homology. In \cite{CM2016} Chowdhury and Mémoli have
sharpened the Dowker Homology Duality
Theorem to a Dowker Homotopy Duality Theorem stating that the
Dowker complexes of \(R\) and \(R^t\) are homotopy equivalent after
geometric realization. That result is a central ingredient in this paper.

In honor of Dowker we name functions \(\Lambda \colon L \times
W \to \rstar\) {\em Dowker dissimilarities}. Forming the Dowker
complexes of the relations \(\Lambda_t\) for \(t \in \rstar\) we
obtain a filtered simplicial complex, the {\em Dowker Nerve}
\(N\Lambda\) of \(\Lambda\), with  
\(N\Lambda_t\) equal to the Dowker complex of
\(\Lambda_t\).

The main result of our work is Theorem \ref{mainresult2} on
sparsification of Dowker nerves.
Here we formulate it in the context of a finite
set \(P\) contained in a metric space \((M,d)\).
Let \(p_0,\dots, ,p_n\) be a farthest point sampling of \(P\)
with insertion radii \(\lambda_0, \dots \lambda_n\). That is,
\(p_0 \in P\) is arbitrary, \(\lambda_0 = \infty\) and 
for each \(0 < k \le n\), the point \(p_k \in P\) is of maximal
distance to 
\(p_0,\dots,p_{k-1}\), and this distance is \(\lambda_{k}\). 
Let \(\varepsilon > 0\) and let \(\Lambda \colon P \times M \to
\rstar\) be the Dowker dissimilarity given by the metric \(d\), that
is, \(\Lambda(p,w) = d(p,w)\). Then the
Dowker Nerve \(N \Lambda\) is equal to the relative \v Cech
complex \(\check \CC(P,M)\) of \(P\) in \(M\) consisting of all balls
in \(M\) centered at 
points in \(P\). Let \([n] = \{0, \dots, n\}\) and
let \(\varphi \colon [n] \to [n]\) be a
function with \(\varphi(0) = 0\) and \(\varphi(k) < k\) and
\begin{displaymath}
  d(p_k, p_{\varphi(k)}) + (\varepsilon + 1)\lambda_k/\varepsilon \le (\varepsilon + 1)\lambda_{\varphi(k)}/\varepsilon
\end{displaymath}
for \(k = 1,\dots, n\).
The {\em Sparse Dowker Nerve} of \(\Lambda\) is the filtered sub-complex 
\(N(\Lambda, \varphi, (\varepsilon + 1)\lambda/\varepsilon)\) of
\(N\Lambda\) with \(N(\Lambda, \varphi, (\varepsilon + 1)\lambda/\varepsilon))_t\) consisting of
subsets \(\sigma \subseteq P\) such that there exists \(w \in M\)
with
\begin{displaymath}
  d(p_k, w) < \min\{ t, (\varepsilon + 1)\lambda_k/\varepsilon,
  (\varepsilon + 1)\lambda_{\varphi(l)}/\varepsilon \}
\end{displaymath}
for every \(k,l \in [n]\). 
\begin{theorem}\label{mainmetricthm}
  The Sparse Dowker Nerve
  \(N(\Lambda, \varphi, (\varepsilon + 1)\lambda/\varepsilon)\) 
  is multiplicatively \((1, 1 + \varepsilon)\)-interleaved with the
  relative \v Cech complex \(\check \CC(P,M)\) of \(P\) in \(M\). 
\end{theorem}
Explicitly, there are
maps \(f_t \colon N\Lambda_t \to N(\Lambda, \varphi, (\varepsilon +
1)\lambda/\varepsilon)_{(1+\varepsilon) t}\) so that if \(g_t \colon 
N(\Lambda, \varphi, (\varepsilon +
1)\lambda/\varepsilon)_{t} \to N\Lambda_t\) are the inclusion maps,
then \(f_t g_t\) and \(g_{(1 + \varepsilon)t} f_t\) are homotopic to
the inclusion of the space of level \(t\) into the space of level \((1
+ \varepsilon) t\) of the filtered simplicial complexes
\(N(\Lambda, \varphi, (\varepsilon +
1)\lambda/\varepsilon)\) and \(N\Lambda\) respectively.
For \(M = \RR^d\) the Sparse Dowker Nerve is a closely related to
the Sparse \v Cech Complex of \cite{SRGeom}. We have implemented both
constructions made them available at GitHub \cite[]{ourCode}. It
turned out that the two constructions are of similar size. We will
leave it for further work to implement a Sparse Dowker Nerve vesion of
the Witness Complex.

Chazal et al. observed in \cite{Chazal2014} that witness complexes and \v Cech
complexes are both
instances of Dowker dissimilarities. The weighted \v Cech complex in
\cite[Definition \(5.1\)]{buchet16efficient} is also an instance of a
Dowker complex.  
Also the filtered clique complex of a finite weighted
undirected simple 
graph 
\((G,w)\) is an
instance of a Dowker nerve: let \(\Pow(G)\) be the set of subsets of
\(G\) and define 
\begin{displaymath}
  \Lambda \colon G \times \Pow(G) \to \rstar, \qquad 
  (v,V) \mapsto
  \begin{cases}
    \diam(V) & \text{if \(v \in V\)} \\
    \infty & \text{otherwise},
  \end{cases}
\end{displaymath}
where \(\diam(V) = \max_{v,v' \in V} w(v,v')\).
Then the Dowker Nerve of \(\Lambda\) is equal to the filtered clique
complex of \(G\). 

For disjoint sets \(L\) and \(W\) a Dowker dissimilarity \(\Lambda \colon L
\times W \to \rstar\) is the same thing as
a weighted simple bipartite graph. On the other hand, a Dowker
dissimilarity of the form \(\Lambda \colon X \times X \to \rstar\) is
the same 
thing as a weighted directed graph with no
multiple edges. In \cite{CM2016} Dowker dissimilarities of this form are
called weighted networks, and their Dowker nerves are studied
thoroughly under the name Dowker complexes. In particular they show
that the persistent homology of the Dowker Nerve of a network is
sensitive to 
the direction its edges.
For example,
for the networks \(A\) and \(B\) in 
Figure \ref{fig:asymmetricnetworks},
with self-loops of weight \(0\), 
the Dowker Nerve of network \(A\) is contractible
while the Dowker Nerve of network \(B\) is homotopic to a circle at
all filtration levels.
\begin{figure}[h]
  \centering
  \begin{displaymath}
    A = 
    \left(
      \begin{tikzcd}
        & 1 \arrow[dr, "0"] & \\
        0 \arrow[ur, "0"] \arrow[rr, "0"] && 2 \\
      \end{tikzcd}
    \right)
    \qquad
    B = 
    \left(
      \begin{tikzcd}
        & 1 \arrow[dr, "0"] & \\
        0 \arrow[ur, "0"]  && 2 \arrow[ll, "0"]\\
      \end{tikzcd}
    \right)
  \end{displaymath}
  `  
  \caption{The Dowker Nerve of network \(A\) is contractible while the Dowker
    Nerve of network \(B\) is homotopic to a circle.}
  \label{fig:asymmetricnetworks}
\end{figure}
Chowdhury and Mémoli also formulate a stability result for homology of Dowker
nerves \cite{CM2016}. We formulate interleaving of Dowker
dissimilarities in such a 
way that their network distance is bounded below by our interleaving
distance. Together with 
functoriality for interleaving distance and the Algebraic Stability
Theorem \cite{Chazal:2009:GSS:1735603.1735622} this implies the
stability result of \cite{CM2016}.
In the context of metric spaces, 
this Stability Theorem
is contained in
\cite{Chazal2014}. 

Imposing conditions on a Dowker dissimilarity of the form
\[\Lambda
\colon X 
\times X \to \rstar\]
we arrive at concepts of independent
interest. Most importantly, \((X,\Lambda)\) is a metric space 
if and only if \(\Lambda\) satisfies
\begin{description}
\item[Finiteness] \(\Lambda(x,y) < \infty\) for all \(x,y \in X\)
\item[Triangle inequality] \(\Lambda(x,z) \le \Lambda(x,y) +
  \Lambda(y,z)\) for \(x,y,x \in X\).
\item[Identity of indiscernibles] \(d(x,y) = 0\) if and only if \(x =
  y\)
\item[Symmetry] \(d(x,y) = d(y,x)\) for all \(x,y \in X\)
\end{description}
Removing some of the above conditions on \(\Lambda\) leads to various
generalizations of metric spaces. 
In particular the situation where \(\Lambda\) only is required to satisfy
the triangle inequality 
has been studied by Lawvere \cite{MR1925933}. He noticed that
\(\rstar\) is a closed symmetric monoidal 
category and 
that when the triangle inequality holds, then \(\Lambda\) gives
\(X\) the structure of a 
category enriched over \(\rstar\). 

Guided by the Functorial Dowker Theorem we have chosen to work with
interleavings in the homotopy category 
instead of on the level of homology groups. We leave it for further
investigation to decide if
the Functorial Dowker Theorem can be extended to homotopy interleavings
in the sense of Blumberg and Lesnick \cite{1705.01690}. 

We extend the usual notion of interleaving between \(\rstar\)-filtered
objects in two
ways. Firstly, we consider interleavings in \(2\)-categories. We were
led to do this because Dowker dissimilarities form a
\(2\)-category, and the proof of the Stability Theorem
is streamlined by working in this
generality. Secondly, following \cite{MR3413628}
we allow
interleaving with respect to order 
preserving functions
of the form \(\alpha \colon \rstar \to \rstar\) satisfying \(t \le
\alpha(t)\) for all \(t\). In this context additive interleaving
corresponds to functions of the form \(\alpha(t) = t + a\) and
multiplicative interleaving corresponds to functions of the form
\(\alpha(t) = ct\). 

After setting terminology and notation, the proof of our main result,
Theorem \ref{mainresult2}, is
a quite simple application of the functorial Dowker Theorem. It consists
of two parts. First we truncate the Dowker dissimilarity associated
to a metric by
replacing certain distances by infinity and show that the truncated Dowker
dissimilarity is interleaved with the original Dowker
dissimilarity. At that point we use the functorial Dowker Theorem. 
Second we give conditions that allow us to sparsify the
Dowker Nerve of the truncated Dowker dissimilarity without changing the filtered homotopy
type. 

The paper is organized as follows: In Section
\ref{sec:contiguitycategory} we 
present the homotopy category of simplicial complexes. In Section
\ref{sec:twocats} we recollect basic terminology about
\(2\)-categories. The main motivation for going to this level of
generality is that interleaving distance in the \(2\)-category \(\dow\) of
Dowker dissimilarities defined in \ref{categorydow} generalizes
network distance from \cite{CM2016}.
Section \ref{sec:interleavings} introduces interleavings in
\(2\)-categories. In Section \ref{sec:multivaluedmaps} we introduce
the \(2\)-category of sets and relations. Section \ref{sec:relcat}
uses the Dowker Nerve construction to define a \(2\)-category with
relations as objects. In Section \ref{sec:dowkerdissimilarities} we
define the \(2\)-category 
of Dowker dissimilarities and introduce the concept of a triangle
relation used as a substitute for the triangle equation for
metric spaces. In Section \ref{sec:stability} we relate interleaving
distance of Dowker dissimilarities to Gromov--Hausdorff distance of
metric spaces. Section \ref{sec:truncated} shows that, under certain
conditions, when some
of the values 
\(\Lambda(l,w)\) in a Dowker dissimilarity are set to infinity the
homotopy type of the Dowker Nerve is only changed up to a certain interleaving.
This is the first step in our proof
of Theorem \ref{mainresult2}. 
In Section \ref{sec:dnerves} we give a criterion ensuring that a certain
sub-complex is homotopy equivalent to the Dowker Nerve of a Dowker
dissimilarity.
Finally in Section \ref{sec:filtereddowkerdissimilarities} we combine
the results of sections \ref{sec:truncated}
and \ref{sec:dnerves} to obtain Theorem \ref{mainresult2}. We also show how
Theorem \ref{mainmetricthm} is a consequence of Theorem \ref{mainresult2} and how the Sparse \v Cech complex \cite{SRGeom} fits into this context.

\section{The Homotopy Category of Simplicial Complexes}
\label{sec:contiguitycategory}

Recall that a simplicial complex \(K = (V,K)\) consists of a vertex
set \(V\) and a set \(K\) of finite subsets of \(V\) with the property that
if \(\sigma\) is a member of \(K\), then every subset of \(\sigma\) is
a member of \(K\). Given a subset \(V' \subseteq V\) and a simplicial
complex \(K = (V,K)\), we write \(K_{V'}\) for the simplicial complex
\(K_{V'} = (V',K_{V'})\) consisting of subsets of \(V'\) of the form
\(\sigma \cap V'\) for \(\sigma \in K\). The {\em geometric
  realization} of a simplicial 
complex \(K = (V,K)\) is the space \(|K|\) consisting
of all functions \(f \colon V \to \RR\) satisfying:
\begin{enumerate}
\item The support \(\{v \in V \, \mid \, f(v) \ne 0\}\) of \(f\) is a
  member of 
  \(K\) 
\item \(\sum_{v \in V} f(v) = 1\).
\end{enumerate}
If \(V\) is finite, then \(|K|\) is given the subspace topology of the
Euclidean space \(\RR^V\). Otherwise  
\(U \subseteq |K|\) is open if and only if for every finite \(V'
\subseteq V\), the set \(U \cap |K_{V'}|\) is open in \(|K_{V'}|\).

A {\em simplicial map} \(f \colon K \to L\) of simplicial complexes
\(K = (V,K)\) and \(L = (W,L)\) consists of a function \(f \colon V
\to W\) such that
\[f(\sigma) = \{f(v) \, \mid \, v \in \sigma\}\]
is
in \(L\) for every \(\sigma \in K\). Observe that a simplicial map \(f
\colon K \to L\) induces a continuous map \(|f| \colon |K| \to |L|\)
of geometric realizations and that this promotes the geometric
realization to a functor \(|\, \cdot \, | \colon \cx \to \topsp\) from the
category \(\cx\) of simplicial complexes and simplicial maps to the
category \(\topsp\) of topological spaces and continuous maps.
\begin{definition}
  The {\em homotopy category \(\hcx\) of simplicial complexes} has
  the class of
  simplicial complexes as objects. Given simplicial complexes \(K\)
  and \(L\), the morphism set \(\hcx(K,L)\) is the set of homotopy
  classes of continuous maps from the geometric realization of \(K\)
  to the geometric realization of \(L\). Composition in \(\hcx\) is
  given by composition of functions representing homotopy classes.
\end{definition}
We remark in passing that the homotopy category of simplicial
complexes is equivalent to the weak homotopy category of topological spaces.

\section{Background on \(2\)-categories}
\label{sec:twocats}
The material in this section is standard. We have taken it from
\cite{math/9810017}. 
Recall that a \(2\)-category \(\CC\) consists of
\begin{enumerate}
\item A class of objects \(A,B,\dots\),
\item For all objects \(A, B\) a category \(\CC(A,B)\). The objects of
  \(\CC(A,B)\) are the morphisms in \(\CC\) and the morphisms \(\alpha
  \colon f \Rightarrow g\) of \(\CC(A,B)\) are the \(2\)-cells in \(\CC\).
\item For every object \(A\) of \(\CC\) there is an identity morphism
  \(\id_A \colon A \to A\) and an identity \(2\)-cell \(\id_{\id_A}
  \colon \id_A \Rightarrow \id_A\).
\item For all objects \(A\), \(B\) and \(C\) of \(\CC\) there is a
  functor
  \begin{align*}
    \CC(A,B) \times \CC(B,C) &\to \CC(A,C) \\
    (f,g) & \mapsto g \cdot f
  \end{align*}
  which
  is associative and admits the identity morphisms and identity
  \(2\)-cells of \(\CC\) as identities.
\end{enumerate}
\begin{definition}
  Given \(2\)-categories \(\CC\) and \(\DD\), a {\em functor} \(F
  \colon 
  \CC \to \DD\) consists of
  \begin{enumerate}
  \item Function \(F \colon \ob \CC \to \ob \DD\)
  \item Functors \(F \colon \CC(A,B) \to \DD(FA,FB)\)
  \end{enumerate}
  such that \(F(\id_A) = \id_{FA}\) and \(Fg \circ Ff = F(g \circ f)\)
  for \(A\) an object of \(\CC\) and \(f \colon A \to B\) and \(g
  \colon B \to C\) morphisms of \(\CC\).
\end{definition}
\begin{definition}
  Given two functors \(F,G \colon \CC \to \DD\) of \(2\)-categories, a
  {\em transformation} 
  \(\alpha \colon F \to G\) consists of
  \begin{enumerate}
  \item A morphism \(\alpha_A \colon FA \to GA\) for every \(A \in \ob \CC\)
  \item A \(2\)-cell \(\alpha_f \colon Gf \circ \alpha_A \to \alpha_B
    \circ Ff\) for every morphism \(f \colon A \to B\) in \(\CC\).
  \end{enumerate}
  This structure is subject the axioms given by commutativity of the
  following two diagrams: 
  \begin{displaymath}
    \begin{tikzcd}
      & Gg \circ \alpha_B \circ Ff \arrow[dr, "\alpha_g \circ \id_{Ff}"]&\\
      Gg \circ Gf \circ \alpha_A \arrow[ur, "\id_{Gg} \circ \alpha_f"] 
      \arrow[rr, "\alpha_{g \cdot f}"] &&
      \alpha_C \circ Fg \circ Ff
    \end{tikzcd}
  \end{displaymath}
  \begin{displaymath}
    \begin{tikzcd}
      & \sigma_A \arrow[dr, "\id"]& \\
      G(\id_A) \circ \alpha_A \arrow[ur, "\id"]
      \arrow[rr, "\alpha_{\id_A}"] &&
      \sigma_A \circ F(\id A).
    \end{tikzcd}
  \end{displaymath}
\end{definition}
\begin{definition}
  Given two functors \(F,G \colon \CC \to \DD\) of \(2\)-categories,
  and transformations
  \(\alpha, \beta \colon F \to G\), a {\em modification} \(M \colon
  \alpha \to \beta\) consists of a \(2\)-cell
  \begin{displaymath}
    M_A \colon \alpha_A \to \beta_A
  \end{displaymath}
  for every object \(A\) of \(\CC\) such that for every morphism \(f
  \colon A \to B\) of \(\CC\) the following diagram commutes:
  \begin{displaymath}
    \begin{tikzcd}
      Gf \circ \alpha_A 
      \arrow[rr, "\id_{Gf} \circ M_A"]
      \arrow[d, "\alpha_f"]
      &&
      Gf \circ \beta_A
      \arrow[d, "\beta_f"] \\
      \alpha_B \circ Ff
      \arrow[rr, "M_B \circ \id_{Ff}"]
      &&
      \beta_B \circ Ff.
    \end{tikzcd}
  \end{displaymath}
\end{definition}
\begin{definition}
  Given \(2\)-categories \(\CC\) and \(\DD\), the
  {\em functor \(2\)-category} \([\CC, \DD]\) is the \(2\)-category with
  functors \(F \colon \CC \to \DD\) as objects, transformations of
  such functors as morphisms and with \(2\)-cells given by modifications.
\end{definition}

Given a category \(\CC\) we will consider it as a \(2\)-category with
only identity \(2\)-cells. Thus, if \(\CC\) is a category and \(\DD\)
is a \(2\)-category we have defined the functor \(2\)-categories
\([\CC,\DD]\) and \([\DD,\CC]\).

\begin{definition}
  The opposite of a \(2\)-category \(\CC\) is the \(2\)-category
  \(\CC^{\op}\) with the same objects as \(\CC\), with
  \begin{displaymath}
    \CC^{\op}(A,B) = \CC(B,A)
  \end{displaymath}
  and with composition obtained from composition in \(\CC\). 
\end{definition}

\section{Interleavings}
\label{sec:interleavings}

We write \(\rstar\) for the extended set of non-negative real numbers 
and consider it as a partially ordered set. We also consider
\(\rstar\) as a category with object set \([0,\infty]\) and with a
unique morphism \(s \to t\) if and only if \(s \le t\).
\begin{definition}
  Let \(\CC\) be a \(2\)-category.
  The {\em category of filtered objects} in \(\CC\) is the functor
  \(2\)-category \([\rstar, \CC]\).  
  A {\em filtered object} in \(\CC\) is an object \(C \colon \rstar \to
  \CC\) of \([\rstar, \CC]\), that is, \(C\) is a functor from
  \(\rstar\) to \(\CC\). A {\em morphism} \(f \colon C \to C'\) of filtered
  objects in \(\CC\) is a transformation.
\end{definition}
\begin{definition}
  Let \(\CC\) be a \(2\)-category and let \(\alpha \colon \rstar \to
  \rstar\) be 
  a functor under the identity, that is, order preserving function
  satisfying \(t \le 
  \alpha(t)\) for all \(t \in \rstar\).
  \begin{enumerate}
  \item
    The the pull-back functor \(\alpha^* \colon
    [\rstar,\CC] \to [\rstar,\CC]\) is the functor taking
    a filtered object \(C \colon \rstar \to
    \CC\) in \(\CC\) to the filtered object \(\alpha^* C = C \circ \alpha\).
  \item
    The {\em unit of the functor \(\alpha^* \colon [\rstar,\CC] \to
      [\rstar,\CC]\)} is the natural transformation \(\alpha_*
    \colon \id \to \alpha^*\) defined by
    \begin{displaymath}
      \alpha_{*C}(t) = C(t \le \alpha(t)) \colon C(t) \to \alpha^*(C)(t) .
    \end{displaymath}
  \end{enumerate}
\end{definition}
\begin{definition}
  Let \(C\) and \(C'\) be filtered objects in a \(2\)-category \(\CC\) and
  let \(\alpha, \alpha' \colon \rstar \to \rstar\) be 
  functors under the identity.
  \begin{enumerate}
  \item 
    An \((\alpha,\alpha')\)-interleaving between \(C\) and \(C'\) is a
    pair \((F,F')\) of morphisms \(F \colon C \to \alpha^*C'\) and \(F'
    \colon C' \to \alpha'^* C\) in \([\rstar,\CC]\) such that there
    exist \(2\)-cells
    \[(\alpha' \circ 
      \alpha)_* \to (\alpha^* F') \circ F
      \quad \text{and} \quad
      (\alpha \circ \alpha')_* \to
    (\alpha'^* F) \circ F'.\] 
  \item 
    We say that \(C\) and \(C'\) are {\em
      \((\alpha,\alpha')\)-interleaved} if there exists an
    \((\alpha,\alpha')\)-interleaving between \(C\) and \(C'\). 
  \end{enumerate}
\end{definition}
The following results appear in \cite[Proposition 2.2.11 and Proposition 2.2.13]{MR3413628}.
\begin{lemma}[Functoriality]
  \label{inducedinterleaving}
  Let \(C\) and \(C'\) be filtered objects in a \(2\)-category \(\CC\),
  let \(\alpha, \alpha' \colon \rstar \to \rstar\) be 
  functors under the identity and let \(H \colon \CC \to \DD\)
  be a functor of \(2\)-categories. 
  If \(C\) and \(C'\) are \((\alpha,\alpha')\)-interleaved, then the
  filtered objects \(H C\) and \(H
  C'\) in \(\DD\) are
  \((\alpha,\alpha')\)-interleaved. 
\end{lemma}
\begin{lemma}[Triangle inequality]
  Let \(C\), \(C'\) and \(C''\) be filtered objects in a \(2\)-category
  \(\CC\). If \(C\) and \(C'\) are \((\alpha,\alpha')\)-interleaved
  and \(C'\) and \(C''\) are \((\beta,\beta')\)-interleaved, then
  \(C\) and \(C''\) are \((\beta \alpha, \alpha'
  \beta')\)-interleaved.  
\end{lemma}

\section{Relations}
\label{sec:multivaluedmaps}

\begin{definition}
  Let \(X\) and \(Y\) be sets.
  A {\em relation} \(R \colon X \leftrightarrows Y\) is a subset
  \(R \subseteq X \times Y\).
\end{definition}
\begin{definition}
  We define a partial order on the set of relations between \(X\)
  and \(Y\) by set inclusion. That is, for relations \(R \colon X
  \leftrightarrows Y\) and \(R' 
  \colon X \leftrightarrows Y\), we have \(R \le R'\) if and only if
  \(R\) contained in the subset of \(R'\) of \(X \times Y\).
\end{definition}
\begin{definition}
  Given two relations \(R \colon X \leftrightarrows Y\) and
  \(S \colon Y \leftrightarrows Z\), their composition 
  \[S \circ R \colon X
  \leftrightarrows Z\] is
  \begin{displaymath}
    S \circ R = \{ (x,z) \in X \times Z \, \mid \,
      \exists \, y \in Y : 
      (x,y) \in R \text{ and } (y,z) \in S, 
     \}.
  \end{displaymath}
\end{definition}
\begin{definition}
  The \(2\)-category \(\setrel\) of sets and relations has as objects
  the class of 
  sets and as morphisms the class of relations. The
  \(2\)-cells are given by the inclusion partial order on
  the class of relations. Composition of morphisms is
  composition of relations and composition of \(2\)-cells is given
  by composition of inclusions. The identity morphism on the set
  \(X\) is the diagonal
  \begin{displaymath}
    \Delta_X = \{(x,x)\, \mid \, x \in X\}.
  \end{displaymath}
  The identity \(2\)-cell on a relation \(R\) is the identity
  inclusion \(R \le R\).
\end{definition}
\begin{definition}
  The transposition functor \(T \colon \setrel \to \setrel^{\op}\) is
  defined by \(T(X) = X\),
  \begin{displaymath}
    T(R) = R^t = \{(y,x) \, \mid \, (x,y) \in R\}
  \end{displaymath}
  and \(T(i) = i^t\), where \(i^t \colon R^t \to S^t\) takes \((y,x)\)
  to \((z,w)\) when \((w,z) = i(x,y)\).
\end{definition}

\begin{definition}\label{definecorresponodence}
  A {\em correspondence} \(C \colon X \leftrightarrows Y\) is a
  relation such that:
  \begin{enumerate}
  \item for every \(x \in X\) there exists \(y \in Y\) so that \((x,y)
    \in C\) and
  \item for every \(y \in Y\) there exists \(x \in X\) so that \((x,y)
    \in C\).
  \end{enumerate}
\end{definition}
\begin{lemma}
  A relation \(C \colon X \leftrightarrows Y\) is a
  correspondence if and only if there exists a relation \(D
  \colon Y \leftrightarrows X\) so that \(\Delta_X \le D \circ C\) and
  \(\Delta_Y \le C \circ D\).
\end{lemma}
\begin{proof}
  By definition of a correspondence, for every \(x \in X\), there
  exists \(y \in Y\) so that \((x,y) \in C\).  This means that
  \(\Delta_X \subseteq  C^t \circ C\), where
  \begin{displaymath}
    C^t \circ C = \{ (x,z) \in X \times X \, \mid \,
    \exists \, y \in Y : 
    (x,y) \in C \text{ and } (y,x) \in C^t 
    \}.
  \end{displaymath}
  Reversing the roles of \(C\) and \(C^t\) we get the inclusion
  \(\Delta_Y \subseteq C \circ C^t\).
  Conversely, if \(C\) and \(D\) are relations with \(\Delta_Y
  \subseteq C \circ D\), then for every \(y 
  \in Y\), the element \((y,y)\) is contained in \(C \circ D\). This means
  that there exists \(x \in X\) so that \((x,y) \in C\), and \((y,x)
  \in D\). In particular, for every \(y \in Y\), there exists \(x \in
  X\) so that \((x,y) \in C\). Reversing the roles of \(C\) and
  \(D\) we get that for every \(x \in X\) there exists \(y \in Y\) so
  that \((x,y) \in C\).
\end{proof}

\section{The category of relations}
\label{sec:relcat}
We start by recalling Dowker's definition of the nerve of a
relation. (Called the complex \(K\) in \cite[Section 1]{MR0048030}.)
\begin{definition}
  Let \(R \subseteq X \times Y\) be a relation. The {\em nerve} of
  \(R\) is the simplicial complex
  \begin{displaymath}
    NR = \{ \text{ finite } \sigma \subseteq X \, \mid \, \exists
    \text{ \(y \in Y\) with \((x,y) \in R\) for all \(x \in
      \sigma\)}\}.  
  \end{displaymath}
\end{definition}
\begin{example}
  Let \(X\) be a space, and let \(Y\) be a cover of \(X\). In particular
  every element \(y \in Y\) is a subset of \(X\). Let \(R\) be the relation 
  \(R \subseteq X \times Y\) consisting of pairs \((x,y)\) with \(x \in y\). 
  A direct inspection reveals that the nerve of \(R\) is equal to
  the Borsuk Nerve of the cover \(Y\). 
\end{example}

\begin{definition}
  The \(2\)-category \(\relcat\) of relations has as objects
  the class 
  of relations. A morphism \(C \colon R \to R'\) in \(\relcat\)
  between 
  relations \(R \subseteq X \times Y\) and \(R' \subseteq X' \times
  Y'\) consists of a relation \(C \subseteq X \times X'\) such that
  for every \(\sigma \in NR\), the set 
  \begin{displaymath}
    (NC)(\sigma) = 
      \{ x' \in X' \, \mid \, \text{ there exists } x \in \sigma \text{
        with } (x,x') \in C\} 
  \end{displaymath}
  is an element \((NC) (\sigma) \in NR'\) of the nerve of \(R'\). In
  particular \((NC) (\sigma)\) is finite and non-empty.
  The class of \(2\)-cells in \(\relcat\) is the class of inclusions
  \(R \subseteq S\) for \(R,S \subseteq X \times Y\). Composition in
  \(\relcat\) is given 
  by composition of relations.
\end{definition}
\begin{lemma}
  Let \(C_1, C_2 \colon R \to R'\) be morphisms in \(\relcat\). If
  there exists a \(2\)-cell \(\alpha \colon C_1 \to C_2\), then
  the simplicial maps \(NC_1\) and \(NC_2\) are contiguous. In
  particular, their geometric realizations are homotopic maps.
\end{lemma}
\begin{proof}
  Let \(\sigma \in NR\). Since \(C_1 \subseteq C_2\), we have an
  inclusion
  \[(NC_1)(\sigma) \subseteq (NC_2)(\sigma),\]
  and thus
  \((NC_2)(\sigma) \in NR'\) implies
  \begin{displaymath}
    (NC_1)(\sigma) \cup (NC_2)(\sigma) = (NC_2)(\sigma) \in NR'.
  \end{displaymath}
  This shows that \(NC_1\) and \(NC_2\) are contiguous.
  For the statement about contiguous maps having homotopic
  realizations see \cite[Lemma 2, p. 130]{Spanier}.
\end{proof}
\begin{definition}
  The {\em nerve functor} \(N \colon \relcat \to \hcx\) is the functor
  taking a relation \(R\) to its nerve \(NR\) and taking a morphism
  \(C \colon R \to R'\) in \(\relcat\) to the morphism \(|NC| \colon
  |NR| \to |NR'|\) in \(\hcx\).   
\end{definition}
Let us emphasize that if \(\alpha \colon C_1 \to C_2\) is a \(2\)-cell in
\(\relcat\), then \(|NC_1| = |NC_2|\) in \(\hcx\). 

\section{Filtered Relations and Dowker dissimilarities}
\label{sec:filtrel}
\label{sec:dowkerdissimilarities}

\begin{definition}
  A {\em filtered relation} is a functor from \(\rstar\) to \(\relcat\).
  We define the \(2\)-category of filtered relations to
  be the 
  \(2\)-category \(\filrel\) of functors from \(\rstar\) to \(\relcat\). 
\end{definition}
\begin{definition}
  The {\em filtered nerve functor} is the functor
  \[N \colon \filrel
    \to \filhtop\]
  from the \(2\)-category of filtered relations to the
  category of homotopy filtered spaces taking \(X \colon \rstar \to
  \relcat\) to the composition
  \begin{displaymath}
    \rstar \xto X \relcat \xto N \htop.
  \end{displaymath}
\end{definition}
From Lemma \ref{inducedinterleaving} we get:
\begin{corollary}\label{inducedcorrolaryfilteredrel}
  If \(R\) and \(R'\) are  \((\alpha,\alpha')\)-interleaved filtered
  relations, then \(NR\) and \(NR'\) are 
  \((\alpha,\alpha')\)-interleaved filtered simplicial complexes.  
\end{corollary}
\begin{definition}\label{filtereddowkermorphism}
  A {\em Dowker dissimilarity} \(\Lambda\) consists of two sets \(L\) and \(W\)
  and a function \(\Lambda \colon L \times W \to \rstar\).
  Given \(t \in \rstar\), we let
  \begin{displaymath}
    \Lambda_t =
    \{(l,w) \in L \times W \, \mid \, \Lambda(l,w) < t\}
  \end{displaymath}
  considered as an object of the category \(\relcat\) of relations,
  and given \(s \le t\) in 
  \(\rstar\) we let 
  \begin{displaymath}
    \Lambda_{s \le t} = \Delta_{L}
  \end{displaymath}
  considered as a morphism \(\Lambda_{s \le t} \colon \Lambda_s
  \to \Lambda_t\) in \(\relcat\).
\end{definition}
\begin{definition}
  The {\em filtered relation associated to} a Dowker dissimilarity
  \(\Lambda \colon L \times W \to \rstar\) is the functor
  \begin{displaymath}
    \Lambda \colon \rstar \to \relcat
  \end{displaymath}
  taking \(t \in \rstar\) to the relation \(\Lambda_t\) and taking \(s
  \le t\) in \(\rstar\) to the morphism \(\Lambda_{s \le t}\) in
  \(\relcat\).
\end{definition}
\begin{definition}
  Let \(\Lambda \colon L \times W \to \rstar\) and
  \(\Lambda' \colon L' \times W' \to \rstar\) be Dowker dissimilarities.
  A morphism \(C \colon \Lambda \to \Lambda'\) of filtered relations
  is a {\em morphism of Dowker
  dissimilarities} if there exists a relation
  \(C \subseteq L \times L'\) so that \(C_t = C \colon
  \Lambda_t \to \Lambda'_t\) for every \(t \in \rstar\).
\end{definition}
\begin{definition}\label{categorydow}
  The {\em \(2\)-category \(\dow\) of Dowker dissimilarities} is the
  \(2\)-category 
  with Dowker 
  dissimilarities as objects and morphisms of Dowker dissimilarities
  as morphisms. 
  Given morphisms \(C_1, C_2 \colon \Lambda \to \Lambda'\) of Dowker
  dissimilarities, we define the set of \(2\)-cells \(\alpha \colon C_1
  \to C_2\) in \(\dow\) by letting  \(\dow(C_1,C_2) = \filrel(C_1,C_2)\).
\end{definition}
\begin{definition}
  Let \(\Lambda \colon L \times W \to \rstar\) be a Dowker
  dissimilarity.
  The {\em Dowker Nerve} \(N \Lambda\) of \(\Lambda\) is the
  filtered nerve of the underlying filtered relation.
\end{definition}

Note that the Dowker Nerve is filtered by
inclusion of sub-complexes, 
that is, if \(s \le t\), then \(N\Lambda_{s \le t} \colon N\Lambda_s
\to N\Lambda_t\) is an inclusion of simplicial complexes.
\begin{definition}\label{definecovrad}
  The {\em cover radius} of a Dowker dissimilarity
  \[\Lambda \colon L
    \times W \to \rstar\]
  is
  \begin{displaymath}
    \rho_\Lambda = \sup_{w \in W} \inf_{l \in L} \Lambda(l,w).
  \end{displaymath}
\end{definition}
\begin{definition}
  Let \(\Lambda \colon L \times W \to \rstar\) be a Dowker
  dissimilarity. Given \(l \in L\) and \(t > 0\), the {\em
    \(\Lambda\)-ball of radius \(t\) centered at \(l\)} is
  \begin{displaymath}
    B_{\Lambda}(l,t) = \{ w \in W\, \mid \, \Lambda(l,w) < t\}.
  \end{displaymath}
\end{definition}

\begin{example}\label{kmeans_example}
  Let \((M, d)\) be a metric space and \(L\) and \(W\) be subsets of
  \(M\).  Then the restriction \(\Lambda \colon L \times W \to \rstar\) of
  \(d\) to \(L \times W\) is a Dowker
  dissimilarity. Its cover radius \(\rho_{\Lambda} = \sup_{w \in W} \inf_{l
    \in L} d(l,w)\) is the directed Hausdorff distance from \(W\) to
  \(L\). The Dowker Nerve of \(\Lambda\) is the composite
  \begin{displaymath}
    \rstar \xto \Lambda \relcat \xrightarrow N\cx
  \end{displaymath}
  taking \(t \in \rstar\) to
  \begin{displaymath}
    \{ \text{ finite } \sigma \subseteq L \, \mid \, \text{ there
    exists \(w \in W\) with \(d(l,w) < t \) for all \(l \in \sigma\)}\}.    
  \end{displaymath}
  If \(L = W = M\), then the \(\Lambda\)-ball of radius \(t\) centered at
  \(l\) is the 
  usual open ball in \(M\) of radius \(t\) centered at \(l\) and the
  Dowker Nerve of \(\Lambda\) is equal to the \v Cech complex \(\cech(M)\).
\end{example}

\begin{lemma}
  Let \(\Lambda \colon L \times W \to \rstar\) be a Dowker
  dissimilarity. Given \(t > 0\), the nerve \(N\Lambda_t\) is
  isomorphic to the Borsuk Nerve of the cover of the set
  \begin{displaymath}
    \bigcup_{l \in L} B_{\Lambda}(l,t)
  \end{displaymath}
  by
  balls \(B_{\Lambda}(l,s)\) of radius \(s \le t\) centered at points in \(L\). 
\end{lemma}
Corollary \ref{inducedcorrolaryfilteredrel} gives:
\begin{corollary}\label{nerveinterleavedcorr}
  If \(\Lambda \colon L \times W \to \rstar\) and \(\Lambda' \colon L'
  \times W' \to \rstar\) are \((\alpha,\alpha')\)-interleaved Dowker
  dissimilarities, then \(N\Lambda\) and \(N\Lambda'\) are
  \((\alpha,\alpha')\)-interleaved filtered simplicial complexes.  
\end{corollary}
\begin{definition}
  Let \(\Lambda \colon L \times W \to \rstar\) be a Dowker
  dissimilarity. The {\em Rips complex} of \(\Lambda\) is the
  filtered simplicial complex \(R\Lambda\) defined by
  \begin{displaymath}
    (R\Lambda)(t) = \{ \text{finite } \sigma \subseteq L \, \mid \,
    \text{every \(\tau \subseteq \sigma\) with \(|\tau| \le 2\) is
      in \((N\Lambda)(t)\)}\}.
  \end{displaymath}
\end{definition}
\begin{corollary}
  If \(\Lambda \colon L \times W \to \rstar\) and \(\Lambda' \colon L'
  \times W' \to \rstar\) are \((\alpha,\alpha')\)-interleaved Dowker
  dissimilarities, then \(R\Lambda\) and \(R\Lambda'\) are
  \((\alpha,\alpha')\)-interleaved filtered simplicial complexes.  
\end{corollary}
\begin{proof}
  Use Corollary \ref{nerveinterleavedcorr} and the fact that the Rips
  complex depends functorially on the one skeleton of the Dowker Nerve.
\end{proof}
The following definition is an instance of the generalized
inverse in \cite{MR3072795}.
\begin{definition}\label{generalizedinverse}
  Let \(\alpha \colon \rstar \to \rstar\) be order preserving with
  \[\lim_{t \to \infty}\alpha(t) \infty.\] 
  The
  generalized inverse function \(\alpha^{\leftarrow} \colon \rstar \to
  \rstar\) is the order preserving function 
  \begin{displaymath}
    \alpha^{\leftarrow}(s) = \inf\{t \in \rstar \, \mid \, \alpha (t) \ge s\}.
  \end{displaymath}
\end{definition}
\begin{lemma}
  Given a Dowker dissimilarity \(\Lambda \colon L
  \times W \to \rstar\) and an order preserving function \(\alpha \colon
  \rstar \to \rstar\), the filtered relation associated to the Dowker
  dissimilarity \(\Lambda\) given as the composite function
  \begin{displaymath}
    L \times W \xto \Lambda \rstar \xto {\alpha^{\leftarrow}} \rstar,
  \end{displaymath}
  is equal to \(\alpha^* \Lambda\).
\end{lemma}
\begin{definition}\label{trianglerelation}
  A {\em triangle relation} for a Dowker dissimilarity
  \[\Lambda \colon
    L \times W \to \rstar\]
  is a relation \(T \subseteq L \times W\)
  with the following properties:
  \begin{enumerate}
  \item For every \(w \in W\) there exists \(l \in L\) so that \((l,w)
    \in T\).
  \item 
    For all \((l,w) \in T\) and \((l',w') \in L
    \times W\), 
    the triangle inequality
    \begin{displaymath}
      \Lambda(l',w') \le \Lambda(l',w) + \Lambda(l,w') + \Lambda(l,w)
    \end{displaymath}
    holds.
  \end{enumerate}
\end{definition}
\begin{remark}
\hspace{0cm}
  \begin{enumerate}
  \item 
    If \(\Lambda_M \colon M \times M \to \rstar\) satisfies the triangle
    inequality
    \begin{displaymath}
      \Lambda_M(x,z) \le \Lambda_M(x,y) + \Lambda_M(y,z)
    \end{displaymath}
    for all \(x,y,z \in Z\), then every relation \(T \subseteq M \times
    M\) satisfies part \((2)\) of Definition
    \ref{trianglerelation}. Moreover, if \(L\) and \(W\) are subsets of
    \(M\) and \(\Lambda \colon L \times W \to M\) is the restriction of
    \(\Lambda_M\) to \(L \times W\), then every relation \(T \subseteq L
    \times W\) satisfies part \((2)\) of Definition
    \ref{trianglerelation}.
  \item 
    Given a Dowker dissimilarity \(\Lambda \colon L \times W \to
    \rstar\) so that the set \(\Lambda(L\times \{w\})\) has a least
    upper bound for every \(w \in 
    W\), there exists a triangle relation \(T\)
    for \(\Lambda\) 
    consisting of the pairs \((l,w)\) satisfying \(\Lambda(l',w) \le
    \Lambda(l,w)\) for all \(l' \in L\).
  \end{enumerate}

\end{remark}

\section{Stability and Interleaving Distance }
\label{sec:stability}

The functoriality of interleaving implies that all functorial
constructions are stable with respect to interleaving. In this section
we relate interleaving distance of Dowker dissimilarities to
Gromov--Hausdorff distance of \cite{MR0401069, MR623534} 
and to the
network distance defined in \cite{CM2016}

\begin{definition}
  Let \(C\) and \(C'\) be filtered objects in a \(2\)-category
  \(\CC\).
  \begin{enumerate}
  \item Given \(a,a' \in \rstar\) we say that the filtered objects
    \(C\) and \(C'\) are 
    {\em additively \((a,a')\)-interleaved} if they are
    \((\alpha,\alpha')\)-interleaved for the functions \(\alpha(t) = a
    + t\) and \(\alpha'(t) = a' + t\).
  \item Let
    \begin{displaymath}
      A(C,C') = \{a \in \rstar \, \mid \, \text{\(C\) and \(C'\) are
        additively \((a,a)\)-interleaved}\}.
    \end{displaymath}
    The {\em interleaving distance} of \(C\) and \(C'\) is
    \begin{displaymath}
      \intdist(C,C') =
      \begin{cases}
        \inf A(C,C') & \text{if \(A(C,C') \ne \emptyset\)} \\
        \infty & \text{otherwise.}
      \end{cases}
    \end{displaymath}
  \end{enumerate}
\end{definition}

\begin{definition}
  A {\em non-negatively weighted network} is a pair \((X, \omega_X)\) of a set \(X\) and a
  weight function \(\omega_X \colon X \times X \to [0,\infty)\).
\end{definition}
\begin{definition}
  Let \(\omega_X \colon X \times X \to
  [0,\infty)\) and \(\omega_{X'} \colon X' \times X' \to [0,\infty)\)
  be non-negatively weighted networks and let \(C \subseteq X \times X'\). The
  {\em distortion of 
    \(C\)} is
  \begin{displaymath}
    \dis(C) = \sup_{(x,x'),\ (y,y') \in C} |\omega_{X}(x,y) - \omega_{X'}(x',y')|.
  \end{displaymath}
\end{definition}
Recall from \ref{definecorresponodence} that \(C \subseteq X \times
X'\) is a correspondence if the 
projections of \(C\) on both \(X\) and \(X'\) are surjective.
\begin{definition}
  Let \(\omega_X \colon X \times X \to
  [0,\infty)\) and \(\omega_{X'} \colon X' \times X' \to [0,\infty)\)
  be non-negatively weighted networks and let \(\mathcal R\) be the set of
  correspondences \(C \subseteq X \times X'\). The {\em network
    distance} between \(X\) and \(X'\) is
  \begin{displaymath}
    \netwdist(X,X') = \frac 12 \inf_{C \in \mathcal R} \dis(C).
  \end{displaymath}
\end{definition}

The Stability Theorem \cite[Proposition 15]{CM2016}
for networks is a consequence
of functoriality of interleaving distance, the Algebraic Stability Theorem
for bottleneck distance \cite[Theorem 4.4]{HardStability} and the following
result:
\begin{proposition}\label{stabilityresult}
  Let \(\omega_X \colon X \times X \to [0,\infty)\) and 
  \(\omega_{X'} \colon X' \times X' \to [0,\infty)\) be networks, and write 
  \begin{displaymath}
    \Lambda \colon X \times X \to \rstar
    \quad \text{and} \quad
    \Lambda' \colon X' \times X' \to \rstar 
  \end{displaymath}
  for the corresponding Dowker dissimilarities with
  \(\Lambda(x,y) =
  {\omega_X(x,y)}\)
  and
  \(\Lambda'(x',y') =
  {\omega_{X'}(x',y')}\).
  Then
  \begin{displaymath}
    \intdist(\Lambda, \Lambda') \le 2 \netwdist(X,X').
  \end{displaymath}
\end{proposition}
\begin{proof}
  We have to show that \(\intdist(\Lambda, \Lambda') \le \dis(C)\) for
  every correspondence \(C \subseteq X \times X'\). So let \(C
  \subseteq X \times X'\) be a correspondence and let \(a >
  \dis(C)\). By definition of \(\dis(C)\), for all \((l,l')\) and
  \((w, w')\) in \(C\) we have
  \begin{displaymath}
    | \omega_{X}(l,w) - \omega_{X'}(l',w')| < a.
  \end{displaymath}
  Defining \(\alpha \colon \rstar \to \rstar\) by \(\alpha(t) = t +
  a\), by symmetry, it suffices to show that \(C\) defines a morphism
  \begin{displaymath}
    C \colon \Lambda \to \alpha^* \Lambda'.
  \end{displaymath}
  That is, we have to show that if \(\sigma \in \Lambda_t\), then
  \((NC)(\sigma) \in \Lambda'_{\alpha t}\). So suppose that \(w \in
  X\) satisfies \(\Lambda(l,w) < t\) for all \(l \in \sigma\). Since \(C\)
  is a correspondence we can pick \(w' \in X'\) so that \((w,w') \in
  C\). By definition of \(NC\), for every \(l' \in (NC)(\sigma)\),
  there exists 
  \(l \in \sigma\) 
  so that \((l,l') \in C\). By definition of distortion distance this gives
  \begin{displaymath}
    \Lambda'(l',w') = \omega_{X'}(l',w') < a + \omega_X(l,w) = a +
    \Lambda(l,w) < a + t = \alpha t.
  \end{displaymath}
  We conclude that \(\sigma \in N\Lambda_t\) implies \((NC)(\sigma)
  \in N\Lambda'_{\alpha t}\) as desired.
\end{proof}
The Stability Theorem \cite[Theorem 5.2]{MR3275299}
for metric spaces is a consequence
of functoriality of interleaving distance, the Algebraic Stability Theorem
for bottleneck distance \cite[Theorem 4.4]{HardStability} and the following
result:
\begin{corollary}
  Let \((M,d)\) and \((M',d')\) be metric spaces, and write
  \begin{displaymath}
    \Lambda \colon M \times M \to \rstar
    \quad \text{and} \quad
    \Lambda' \colon M' \times M' \to \rstar 
  \end{displaymath}
  for the corresponding Dowker dissimilarities with \(\Lambda(p,q) =
  d(p,q)\) and \(\Lambda'(p',q') =
  d'(p',q')\). Then
  \begin{displaymath}
    \intdist(\Lambda, \Lambda') \le 2 d_{GH}(M,M').
  \end{displaymath}  
\end{corollary}
\begin{proof}
  By \cite[Theorem 7.3.25]{MR1835418} the Gromov--Hausdorff distance
  of the metric spaces \((M,d)\) and \((M',d')\) agrees with their
  network distance when we consider them as non-negatively weighted
  networks. That is, 
  \(d_{GH}(M,M') = \netwdist(M,M')\).
  The result now follows from Proposition \ref{stabilityresult}.
\end{proof}

\section{Truncated Dowker Dissimilarities}
\label{sec:truncated}

\begin{definition}\label{definsertionfct}
  Let \(\Lambda \colon L
  \times W \to \rstar\) be a Dowker dissimilarity, let \(T \subseteq L
  \times W\) be a triangle relation for \(\Lambda\)
  and let \(\beta
  \colon \rstar \to \rstar\) be an order preserving function.
  A {\em \(T\)-insertion function for \(\Lambda\) of resolution at most
    \(\beta\) 
  }
  is a function \(\lambda 
  \colon W \to \rstar\) with the property that for every \(t \in
  \rstar\) and for every 
  \((l,w) \in T\)  there exists \(w_0 \in W\) so that
  \begin{displaymath}
    \Lambda(l,w_0) \le \beta(t) < \lambda(w_0).
  \end{displaymath}
\end{definition}

\begin{example}
  Recall the Dowker dissimilarity \(\Lambda \colon L \times W \to \rstar\)
  from Example~\ref{kmeans_example} for two subsets \(L\) and \(W\) of
  a metric space \((M, d)\) and let \(\beta\) be any order preserving
  function with \(\beta (t) > \rho_\Lambda\). Then for every \(T
  \subseteq L \times W\) the function \(\lambda \equiv \infty\)
  is a \(T\)-insertion function for \(\Lambda\) of resolution at most \(\beta\).
\end{example}
In the following definition we use the generalized inverse from Definition
\ref{generalizedinverse}. 
\begin{definition}\label{definetruncation}
  Let \(\Lambda \colon L \times W \to \rstar\) be a Dowker
  dissimilarity with a triangle relation \(T\) and a \(T\)-insertion
  function \(\lambda\) of resolution 
  at most \(\beta\) for an order preserving \(\beta \colon \rstar \to
  \rstar\) with
  \begin{displaymath}
    \lim_{t \to \infty} \beta(t) = \infty.
  \end{displaymath}
  Given an order
  preserving function \(\alpha \colon \rstar \to \rstar\), satisfying
  that 
  \(\alpha(t) \ge t + \beta(t)\) for all \(t\), the {\em
    \((\lambda,\alpha,\beta)\)-truncation} of \(\Lambda\) is the Dowker
  dissimilarity \(\Lambda^{(\lambda,\alpha,\beta)} \colon L \times W \to
  \rstar\) defined by
  \begin{displaymath}
    \Lambda^{(\lambda,\alpha,\beta)}(l, w) =
    \begin{cases}
      \Lambda(l,w) & \text{if \(\Lambda(l,w)
        \le \alpha \beta^{\leftarrow} \lambda(w)\)} \\
      \infty & \text{otherwise}.
    \end{cases}
  \end{displaymath}
\end{definition}

\begin{lemma}\label{morphismintotruncation}
  Let \(\Lambda \colon L \times W \to \rstar\) be a Dowker dissimilarity
  with a triangle relation \(T\) and a \(T\)-insertion function \(\lambda\) of
  resolution at most \(\beta \colon \rstar \to \rstar\).
  If \(\alpha \colon \rstar \to \rstar\) is an order preserving
  function satisfying 
  \begin{displaymath}
    \alpha(t) \ge t + \beta(t) + \sup \Lambda(T)
  \end{displaymath}
  for all \(t \in \rstar\), then
  \(\Delta_L\) is a morphism \(\Delta_L \colon \Lambda \to
  \alpha^* \Lambda^{(\lambda,\alpha,\beta)}\) of Dowker dissimilarities.
\end{lemma}
\begin{proof}
  Let \(t \in \rstar\) and \(\sigma \in N\Lambda_t\). We need to show
  \(\sigma \in N\Lambda^{(\lambda,\alpha,\beta)}_{\alpha
    t}\). Pick \(w \in W\) with \(\Lambda(l,w) < t\) for all \(l
  \in \sigma\). Since \(T\) is a triangle relation we can pick \(l_0
  \in L\) so that \((l_0,w)\in T\). Since \(\lambda\) is a
  \(T\)-insertion function of
  resolution at most \(\beta\)
  we can pick \(w_0 \in W\) so that
  \begin{displaymath}
    \Lambda(l_0, w_0) \le \beta (t) < \lambda w_0.
  \end{displaymath}
  The triangle inequality for \(T\) now gives
  \begin{displaymath}
    \Lambda(l,w_0) \le \Lambda(l_0, w_0) + \Lambda(l_0, w) + \Lambda(l, w).
  \end{displaymath}
  We have picked \(l_0\), \(w\) and \(w_0\) so that \(\Lambda(l_0, w) \le
  \sup \Lambda (T)\) and also \(\Lambda(l_0, w_0) \le \beta t\).
  If \(l \in \sigma\), then \(\Lambda(l, w) < t\), and thus
  \begin{displaymath}
    \Lambda(l,w_0) < \beta t + \sup \Lambda (T) + t  = \alpha t.
  \end{displaymath}
  From part \((5)\) in \cite[Proposition 1]{MR3072795} the inequality
  \(\beta(t) < \lambda(w_0)\) gives \(t \le \beta^{\leftarrow} \lambda
  (w_0)\). Since \(\alpha\) is
  order preserving we get
  \(\Lambda(l,w_0) < \alpha \beta^{\leftarrow} \lambda w_0\).
  We conclude that \(\sigma \in N\Lambda^{(\lambda, \alpha,
    \beta)}_{\alpha t}\). 
\end{proof}
\begin{proposition}\label{interleavedtruncation}
  Let \(\Lambda \colon L \times W \to \rstar\) be a Dowker dissimilarity
  with an insertion function \(\lambda \colon W \to \rstar\) of
  resolution at most \(\beta \colon \rstar \to \rstar\) and a
  triangle relation \(T \subseteq L \times W\).
  If an order preserving function \(\alpha \colon \rstar \to \rstar\)
  satisfies
  \begin{displaymath}
    \alpha(t) \ge t + \beta(t) + \sup \Lambda(T)
  \end{displaymath}
  for all \(t \in \rstar\), then the Dowker dissimilarities
  \(\Lambda\) and \(\Lambda^{(\lambda,\alpha,\beta)}\) are \((\alpha,
  \id)\)-interleaved. 
\end{proposition}
\begin{proof}
  By Lemma \ref{morphismintotruncation}, the relation \(\Delta_L\) gives
  a morphism
  \[\Delta_L \colon \Lambda \to
    \alpha^* \Lambda^{(\lambda,\alpha,\beta)}\]
  of Dowker
  dissimilarities. Since \(\Lambda(l,w) \le
  \Lambda^{(\lambda,\alpha,\beta)}(l,w)\) for all \((l,w) \in L \times
  W\), the relation \(\Delta_L\) also gives a
  a morphism \(\Delta_L \colon \Lambda^{(\lambda,\alpha,\beta)} \to
  \Lambda\) of Dowker 
  dissimilarities. 
\end{proof}

\section{Sparse Dowker Nerves}
\label{sec:dnerves}

\begin{definition}
  Let \(\Lambda \colon L \times W \to \rstar\) be a Dowker
  dissimilarity and let \(\varphi \colon L \to L\) and \(\lambda \colon L
  \to \rstar\) be functions.
  Given \(\sigma \in N\Lambda_\infty\), the {\em radius}
  of \(\sigma\) is
  \begin{displaymath}
    r(\sigma) = \inf \{t \, \mid \, \sigma \in N\Lambda_t \} .
  \end{displaymath}
  The {\em sparse \((\varphi, \lambda)\)-nerve} of \(\Lambda\) is the
  filtered simplicial complex \(N(\Lambda,\varphi, \lambda)\) defined by
  \begin{displaymath}
    N(\Lambda, \varphi, \lambda)(t) = 
    \{ \sigma \in N\Lambda_t \, \mid \,
    r(\sigma) \le \lambda(\varphi(l)) \text{ for all \(l \in \sigma\)}\}.
  \end{displaymath}
\end{definition}

\begin{proposition}\label{formalpropertiesgivedeformationretract}
  Let \(\Lambda \colon L \times W \to \rstar\) be a Dowker
  dissimilarity and let \(\varphi \colon L \to L\) and \(\lambda \colon L
  \to \rstar\) be functions.
  Suppose there exists \(l_0 \in L\) and an integer \(N \ge 0\) so that
  for all \(l \in L\) and all \(t \in \rstar\): 
  \begin{enumerate}
  \item \(\varphi^N(l) = l_0\).
  \item \(B_\Lambda(l, \lambda(l)) \subseteq B_\Lambda(\varphi(l),
    \lambda(\varphi(l)))\). 
  \item \(B_\Lambda(l, t) = B_\Lambda(l, \lambda(l))\) if \(\lambda(l) \le t\).
  \item \(\lambda(\varphi(l)) \ge \lambda(l)\).
  \end{enumerate}
  Then for every \(t \in \rstar\) the inclusion of \(N(\Lambda,
  \varphi, \lambda)(t)\) in  
  \((N\Lambda)(t)\)
  is a homotopy equivalence.
\end{proposition}
\begin{proof}
  Assumptions \((1), (3)\) and \((4)\) together imply that \(N\Lambda_t =
  N\Lambda_{\lambda(l_0)}\) and \(N(\Lambda, \varphi, \lambda)(t) =
  N(\Lambda, \varphi, \lambda)(\lambda(l_0))\) for \(t \ge \lambda(l_0)\). 
  Thus
  it suffices to prove the claim for \(t \le \lambda(l_0)\). 
  In this situation we will show that the inclusions of
  \begin{displaymath}
    N_t(\Lambda, \varphi, \lambda) = 
    \{ \sigma \in N\Lambda_t \, \mid \,
    t \le \lambda (\varphi( l)) \text{ for all \(l \in \sigma\)}\}
  \end{displaymath}
  in both \(N(\Lambda, \varphi, \lambda)(t)\) and in
  \(N\Lambda_t\)
  are deformation
  retracts.
  For this it suffices to find a map \(f \colon N\Lambda_t \to
  N\Lambda_t\) with the 
  following three properties: firstly both \(f\) and its restriction
  \(f \colon N(\Lambda, \varphi, \lambda)(t) \to N(\Lambda, \varphi, \lambda)(t)\) are
  contiguous to the 
  identity. Secondly we have
  \(f(\sigma) = \sigma\) for 
  every \(\sigma \in N_t(\Lambda, \varphi, \lambda)\), and thirdly \(f(\sigma)
  \in N_t(\Lambda, \varphi, \lambda)\) for every \(\sigma \in N\Lambda_t\).

  For \(t \le \lambda(l_0)\) 
  we use assumption \((1)\) to define a
  function \(f \colon 
  L \to L\) by 
  \begin{displaymath}
    f(l) =
      \varphi^m(l) \text{ for \(m \ge 0\) minimal with \(\lambda(
        \varphi^{m+1}(l))
        \ge t\)}.
  \end{displaymath}
  Given \(\sigma \in N\Lambda_t\) we 
  let \(f(\sigma) =
  \{f(l) \, \mid \, l \in \sigma\}\). By construction, if
  \(\sigma \in N_t(\Lambda, \varphi, \lambda)\), then \(f(\sigma) = \sigma\).
  On the other hand, by construction, \(\lambda(\varphi(
  f(l))) 
  \ge t\) for all \(l \in \sigma\) so
  \(f(\sigma) \in N_t(\Lambda, \varphi, \lambda)\).

  Note that if \(\lambda(\varphi(l)) < t\), then assumption \((2)\) gives
  \begin{displaymath}
    B_\Lambda(l, \lambda(l)) \subseteq B_\Lambda(\varphi(l),
    \lambda(\varphi(l))) \subseteq 
    B_\Lambda(\varphi(l),t), 
  \end{displaymath}
  and together with assumptions \((3)\) and \((4)\) we get
  \begin{displaymath}
    B_\Lambda(l, t) \subseteq B_\Lambda(\varphi(l), t).
  \end{displaymath}
  On the other hand, if \( \lambda(\varphi(l)) \ge t\), then \(f(l) = l\).
  It follows by induction that \(B_\Lambda(l, t) \subseteq
  B_\Lambda(f(l), t)\) for every \(l \in L\). 
  This implies that 
  the map \(f \colon L \to L\) induces simplicial maps
  \(f \colon N \Lambda_t \to N\Lambda_t\) and \(f \colon N(\Lambda,
  \varphi, \lambda)(t) \to N(\Lambda, \varphi, \lambda)(t)\) which are
  contiguous to the respective 
  identity maps. 
\end{proof}

\section{Dowker Dissimilarities On Finite Ordinals}
\label{sec:filtereddowkerdissimilarities}
In this section give a sparse approximation to the Dowker Nerve for 
Dowker dissimilarities of the form
\begin{displaymath}
  \Lambda \colon L \times [n] \to \rstar,
\end{displaymath}
where \([n] = \{0 < 1 < \dots < n\}\).
\begin{definition}\label{defineinsertionandparent}
  Let \(n \ge 0\) be a natural number, let
  \begin{displaymath}
    \Lambda \colon L \times [n] \to \rstar
  \end{displaymath}
  be a Dowker dissimilarity and let \(T \subseteq L \times [n]\) be a
  triangle relation for \(\Lambda\).
  \begin{enumerate}
  \item The {\em domain} of \(T\) is the set
    \begin{displaymath}
      D(T) = \{l \in L \, \mid \, \text{there exists \(k \in [n]\) with \((l,k) \in T\)}\}. 
    \end{displaymath}
  \item The {\em insertion radius of \(k \in [n]\) with respect to
      \(\Lambda\) and \(T\)} is
    \begin{displaymath}
      \lambda_{\Lambda, T}(k) =
      \begin{cases}
        \infty & \text{ if \(k = 0\)} \\
        \sup_{l \in D(T)} \inf_{i \in [k-1]}  \Lambda(l,i) & \text{ if \(k
          > 0\)}.
      \end{cases}
    \end{displaymath}
  \end{enumerate}
\end{definition}

Recall the definition of the cover radius \(\rho_\Lambda\) of a Dowker
dissimilarity in Definition \ref{definecovrad} and the definition of
\(T\)-insertion functions for \(\Lambda\) in \ref{definsertionfct}.
\begin{lemma}\label{lambdaisaninsertionfunction}
  Let \(\Lambda \colon L \times [n] \to \rstar\) be a Dowker
  dissimilarity and let \(\beta \colon \rstar \to \rstar\) be an order
  preserving 
  function with \(\beta(t) \ge \rho_\Lambda\) for all \(t \in
  \rstar\). The insertion radius 
  \(\lambda_{\Lambda , T}
  \colon [n] \to \rstar\) with respect to \(\Lambda\) and \(T\) is a
  \(T\)-insertion function for \(\Lambda\) of  
  resolution at most \(\beta\). 
\end{lemma}
\begin{proof}
  Given \(t \in \rstar\) and \(l \in L\), let \(i \in [n]\) be minimal
  under the condition that \(\Lambda(l,i) \le \beta t\). Then, by definition of
  \(\lambda_{\Lambda, T}\), we have \(\lambda_{\Lambda, T}(i) > \beta
  t\). 
\end{proof}
\begin{definition}\label{defineparentfunction}
  Let \(\Lambda \colon L \times [n] \to \rstar\) be a Dowker
  dissimilarity, let \(T \subseteq  L \times [n]\) be a triangle
  relation for \(\Lambda\) and let \(\beta \colon \rstar \to \rstar\)
  be an order preserving function with \(\beta(t) \ge \rho_\Lambda\)
  for all \(t \in 
  \rstar\). 
  Suppose that \(\lim_{t \to \infty} \beta(t) = \infty\) and let
  \(\alpha \colon \rstar \to \rstar\) be the function 
  \begin{displaymath}
    \alpha(t) = t + \beta(t) + \sup(\Lambda(T))
  \end{displaymath}
  and let \(\lambda \colon [n] \to [n]\) be the function
  \begin{displaymath}
    \lambda(k) = \alpha \beta^{\leftarrow}\lambda_{\Lambda, T}(k)
  \end{displaymath}

  The {\em parent function} \(\varphi \colon
  [n] \to [n]\) is defined by letting \(\varphi(0) = 0\) and
  \begin{displaymath}
    \varphi(k) = \max\{ i \in [k-1] \, \mid
    \, B_{\Lambda}(k, \lambda(k))
    \subseteq B_{\Lambda}(i, 
    \lambda(i)) \text{ and } \lambda(k) \le
    \lambda(i)\}.
  \end{displaymath}
\end{definition}
The following result is about sparsification of truncated
Dowker dissimilarities. We remind that the truncated Dowker
dissimilarity \(\Lambda^{(\lambda_{\Lambda}, \alpha, \beta)}\) comes
from Definition \ref{definetruncation}.
\begin{theorem}\label{mainthm}
  Suppose, in the situation of Definition \ref{defineparentfunction},
  that \(B_{\Lambda^t}(0,\infty) = L\). It we let \(\Gamma =
  (\Lambda^{(\lambda_{\Lambda}, \alpha, \beta)})^t\), then
  the Dowker Nerve
  \(N\Lambda\) of \(\Lambda\) is \((\alpha, \id)\)-interleaved with
  the filtered simplicial complex \(N(\Gamma,
    \varphi, \lambda)\).
\end{theorem}
\begin{proof}
  We first check that Proposition
  \ref{formalpropertiesgivedeformationretract}
  applies to the Dowker dissimilarity \(\Gamma \colon [n] \times L \to \rstar\)
  and the functions \(\varphi\colon [n] \to [n]\)
  and \(\lambda \colon [n] \to \rstar\). By
  construction \(\varphi(0) = 0\) and \(\varphi(k) < k\) for \(k > 0\), so
  \(\varphi^n(k) = 0\) for every \(k \in [n]\). Thus condition
  \((1)\) of \ref{formalpropertiesgivedeformationretract} holds for
  \(\varphi\). By construction
  of \(\varphi\) the assumption that
  \(B_{\Lambda^t}(0,\infty) = L\) implies conditions \((2)\) and \((4)\) of
  \ref{formalpropertiesgivedeformationretract}. Condition \((3)\) of 
  \ref{formalpropertiesgivedeformationretract} holds by construction
  of \(\Lambda^{(\lambda_{\Lambda}, \alpha, \beta)}\). We conclude
  that by Proposition \ref{formalpropertiesgivedeformationretract} the
  filtered simplicial complexes \(N(\Gamma, \varphi, \lambda)\) and
  \(N\Gamma\)
  are homotopy equivalent.
  The functorial Dowker theorem \cite[Corollary 20]{CM2016} implies that
  the filtered simplicial complexes \(N\Gamma\)
  and 
  \(N(\Lambda^{(\lambda_{\Lambda},
    \alpha, \beta)})\) are homotopy equivalent. 
  By Lemma \ref{lambdaisaninsertionfunction} the function 
  \(\lambda_{\Lambda} \colon [n] \to \rstar\) is an insertion function
  for \(\Lambda\), so by
  Proposition \ref{interleavedtruncation} the filtered simplicial
  complexes
  \(N \Lambda\) and
  \(N(\Lambda^{(\lambda_{\Lambda},
    \alpha, \beta)})\)
  are \((\alpha, \id)\)-interleaved. 
\end{proof}
\begin{theorem}\label{mainresult2}
  Let
  \[\Lambda \colon L \times [n] \to \rstar\]
  be a Dowker
  dissimilarity with \(B_{\Lambda^t}(0, \infty) = L\). Let \(T\) be a
  triangle relation for \(\Lambda\) and let \(\beta \colon \rstar \to
  \rstar\) be an order preserving function with \(\lim_{t \to \infty}
  \beta(t) = \infty\). Let \(\alpha \colon \rstar \to \rstar\) be the
  function
  \begin{displaymath}
    \alpha(t) = t + \beta(t) + \sup (\Lambda(T))
  \end{displaymath}
  and let \(\lambda \colon [n] \to [n]\) be the function
  \begin{displaymath}
    \lambda(k) = \alpha \beta^{\leftarrow}\lambda_{\Lambda, T}(k).
  \end{displaymath}
  Let 
  \(\varphi \colon
  [n] \to [n]\) be the parent function defined by letting \(\varphi(0) = 0\) and
  \begin{displaymath}
    \varphi(k) = \max\{ i \in [k-1] \, \mid
    \, B_{\Lambda}(k, \lambda(k))
    \subseteq B_{\Lambda}(i, 
    \lambda(i)) \text{ and } \lambda(k) \le
    \lambda(i)\}.
  \end{displaymath}
  It we let \(\Gamma =
  (\Lambda^{(\lambda_{\Lambda}, \alpha, \beta)})^t\), then
  the Dowker Nerve
  \(N\Lambda^t\) of \(\Lambda^t\) is \((\alpha, \id)\)-interleaved with
  the filtered simplicial complex \(N(\Gamma, \varphi, \lambda)\).
\end{theorem}
\begin{proof}
  By Theorem \ref{mainthm} we have that
  \(N\Lambda\) and \(N(\Gamma, \varphi, \lambda)\)
  are \((\alpha, \id)\)-interleaved. Now use the functorial Dowker
  Theorem to get that the filtered simplicial complexes \(N\Lambda\)
  and \(N\Lambda^t\)  are homotopy equivalent.
\end{proof}
As a special case of Theorem \ref{mainresult2} we get the following result:
\begin{corollary}\label{multiplicativeinterleavingresult1}
  In the situation of Theorem \ref{mainresult2},
  let \(c > 1\), let \(\beta \colon \rstar \to \rstar\) be the function
  \begin{displaymath}
    \beta(t) = \max ((c-1)t, \rho_{\Lambda})
  \end{displaymath}
  and let \(\alpha \colon \rstar \to \rstar\) be the function
  \begin{displaymath}
    \alpha(t) = t + \beta(t) + \sup(\Lambda(T)).
  \end{displaymath}
  The Dowker Nerve
  \(N\Lambda\) of \(\Lambda\) is \((\alpha, \id)\)-interleaved with
  the filtered simplicial complex
  \begin{displaymath}
    N((\Lambda^{(\lambda_{\Lambda}, \alpha, \beta)})^t, \varphi, \lambda).
  \end{displaymath}
\end{corollary}
Specializing even further, we get obtain a variation of the Sparse 
\v Cech complex
of \cite{SRGeom}:
\begin{corollary}\label{sparseDowkernerve}
  Let \((M,d)\) be a metric space, let \(L \subseteq M\) be a compact
  subset, let \(P\) be a finite
  subset of \(M\) and let \([n] \xto p
  P\) be a bijection. Let \(\Lambda \colon M \times [n] \to \rstar\)
  be the function 
  \begin{displaymath}
    \Lambda(x, k) = d(x, p_k),
  \end{displaymath}
  where we write \(p_k = p(k)\). Let \(T \subseteq M \times [n]\)
  be the triangle relation for \(\Lambda\) consisting of the pairs
  \((l,k)\) such that \(d(l,p_k) \le d(l',p_k)\) 
  for every \(l' \in L\).
  Let \(c > 1\), let \(\beta \colon \rstar \to \rstar\) be the function
  \begin{displaymath}
    \beta(t) = \max ((c-1)t, \rho_{\Lambda})
  \end{displaymath}
  and let \(\alpha \colon \rstar \to \rstar\) be the function
  \begin{displaymath}
    \alpha(t) = t + \beta(t) + \sup(\Lambda(T)).
  \end{displaymath}
  For \(\varphi\) and \(\lambda\) as in Definition
  \ref{defineparentfunction}, the Dowker Nerve
  \(N\Lambda^t\) of \(\Lambda^t\) is \((\alpha, \id)\)-interleaved with
  the filtered simplicial complex
  \begin{displaymath}
    N((\Lambda^{(\lambda_{\Lambda}, \alpha, \beta)})^t, \varphi, \lambda)
  \end{displaymath}
  and \(N\Lambda^t\) is additively \((2d_{GH}(L,P), 2d_{GH}(L,P))\)-interleaved
  with the relative \v Cech complex \(\cech(L,M)\) consisting of all
  balls in \(M\) with centers in \(L\).
\end{corollary}
\begin{proof}
  Corollary \ref{multiplicativeinterleavingresult1} gives that 
  \(N\Lambda\) is \((\alpha, \id)\)-interleaved with
  \begin{displaymath}
    N((\Lambda^{(\lambda_{\Lambda}, \alpha, \beta)})^t, \varphi, \lambda). 
  \end{displaymath}

  For second statement note that the stability \ref{stabilityresult}
  implies that the Dowker dissimilarities \(d
  \colon M \times P \to \rstar\) and \(d \colon M \times L \to
  \rstar\) are additively \((2d_{GH}(L,P),
  2d_{GH}(L,P))\)-interleaved. Now use that \(N\Lambda\) is isomorphic
  to the Dowker Nerve of \(d
  \colon M \times P \to \rstar\), and that the Dowker Nerve of
  \(d \colon M \times L \to
  \rstar\) 
  is the relative \v Cech complex \(\cech(L,M)\).
\end{proof}
Finally, we relate the Sparse Dowker Nerve to the Sparse \v Cech
complex of \cite{SRGeom}:
\begin{proposition}\label{sparsecech}
  Let \(d\) be a convex metric on \(\RR^d\) and let \(P\) be a finite
  subset of \(\RR^d\) together with a greedy order \([n] \xto p
  P\). Let the function \(\Lambda \colon \RR^d \times [n] \to \rstar\)
  be given by
  \begin{displaymath}
    \Lambda(x, k) = d(x, p_k),
  \end{displaymath}
  where we write \(p_k = p(k)\).
  Let \(\varepsilon > 0\) and let \(\alpha, \beta \colon \rstar \to
  \rstar\) be the functions \(\beta(t) = \varepsilon t\) and 
  \(\alpha(t) = (1 + \varepsilon)t\). In the notation of Definition
  \ref{defineinsertionandparent}, let \(T = P \times [n]\) and let
  \(\lambda =  
  \lambda_{\Lambda, T}(1+\varepsilon)^2/\varepsilon\). Then the
  filtered simplicial complex
  \begin{displaymath}
    N((\Lambda^{(\lambda, \alpha, 
    \beta)})^t, \id, \lambda)(t)
  \end{displaymath}
  is isomorphic to the filtered simplicial complex \(\{\bigcup_{s <
    t}S^s\}_{t \ge 0}\) 
  obtained from the sparse \v Cech complex \(\{S^t\}_{t \ge 0}\)
  constructed in \cite[Section 4]{SRGeom}. 
\end{proposition}
\begin{proof}
  A subset \(\sigma \subseteq [n]\) is in 
  \begin{displaymath}
    N((\Lambda^{(\lambda, \alpha, 
    \beta)})^t)
  \end{displaymath}
  if and only if there exists \(w \in \RR^d\) so that for all \(l \in
  \sigma\) we have
  \begin{displaymath}
    d(p_l,w) < t 
    \quad \text{and} \quad 
    d(p_l,w) \le \lambda_{\Lambda, T}(l)(1+\varepsilon)/\varepsilon.
  \end{displaymath}
  Moreover
  \begin{displaymath}
    \sigma \in N(((\Lambda^t)^{(\lambda, \alpha, 
      \beta)})^t, \id, \lambda)(t)
  \end{displaymath}
  if and only is there exists \(x \in \RR^d\) so that
  for all \(k,l \in
  \sigma\) we have \(d(p_k,x) < t\) and
  \begin{displaymath}
    d(p_k,x) \le \lambda_{\Lambda, T}(k)(1+\varepsilon)/\varepsilon
    \quad \text{and} \quad 
    d(p_k,x) \le
    \lambda_{\Lambda, T}(l)(1+\varepsilon)^2/\varepsilon.
  \end{displaymath}
  On the other hand, \(\sigma \in S^t\) if and only if there exists
  \(s \le t\) and \(w \in \RR^d\) so that \(w \in b_l(s)\) for all \(l
  \in \sigma\). By the definition of \(b_l(s)\) defined in
  \cite[Section 3]{SRGeom}. This is the case if and only if \(s \le
  t\) and
  \begin{displaymath}
    s \le
    \lambda_{\Lambda, T}(l) (1+\varepsilon)^2/\varepsilon
    \quad \text{and} \quad
    d(p_l,w) \le \min(s, \lambda_{\Lambda, T}(l) (1+\varepsilon)/\varepsilon)
  \end{displaymath}
  for every \(l \in
  \sigma\).
  We conclude that \(\sigma \in S^t\) if and only if there exists \(w
  \in \RR^d\) satisfying
  \(d(p_l,w) \le t\) and
  \begin{displaymath}
    d(p_l,w) \le \lambda_{\Lambda, T}(l)(1+\varepsilon)/\varepsilon
    \quad \text{and} \quad 
    d(p_k,w) \le
    \lambda_{\Lambda, T}(l)(1+\varepsilon)^2/\varepsilon.
  \end{displaymath}
  for all \(k,l \in \sigma\).
\end{proof}

We have not performed any complexity analysis of Sparse Dowker
Nerves. Instead we have made proof-of-concept implementations of
slight variations
of both
the Sparse \v Cech Complex of 
\cite{SRGeom} described in Proposition \ref{sparsecech} and the Sparse
Dowker Nerve described in Corollary \ref{sparseDowkernerve}. 
These implementations come with the same interleaving guarantees, but
for practiacal reasons concerning the miniball algorithm 
we consider complexes that are slightly bigger
than the ones described above. 
We have tested these implementations the following data:
The optical patch data sets
called \(X(300,30)\) and \(X(15,30)\) 
in \cite{MR3715451}, \(6,040\) points from the cyclo-octane
conformation space as analyzed in
\cite{MR2963600} the Clifford data set consisting of \(2,000\)
points on a curve on a torus considered in \cite[Chaper 5]{MR3408277}
and the double torus from \cite{simba}.
Computing the Sparse \v Cech complexes and the Sparse Dowker Nerves on
these data sets with the same interleaving
constant \(c\) the resulting simplicial complexes are almost of
the same size, with the size of the Sparse Dowker Nerve slightly
smaller than the size of the Sparse \v Cech Complex. 
Our implementations, the data sets mentioned above and the scripts
used to run compute persistent homology is available \cite{ourCode}

\section{Conclusion}
\label{sec:conclusion}

We have generalized the Sparse \v Cech construction of \cite{SRGeom}
to arbitrary metric spaces and to a large class of Dowker
dissimilarities. 
The abstract context of Dowker dissimilarities is well suited for
sparse nerve constructions. 
The concepts of filtered relations and strict
\(2\)-categories enable us to easily formulate and prove basic stability
results.  
An implementation of the Sparse Dowker Nerve 
most
similar to the Sparse \v Cech complex is available at GitHub
\cite[]{ourCode}. This implementation is not practical for analysis of
high dimensional data. The current bottleneck is the construction of a
clique  
complex. In further work we will improve this construction and we will
make Sparse Dowker Nerve versions of the Witness Complex.

\printbibliography
\end{document}